\documentclass{gtmon_a}
\pdfoutput=1

\usepackage{graphicx}

\proceedingstitle{Groups, homotopy and configuration spaces (Tokyo
  2005)}
\conferencestart{5 July 2005}
\conferenceend{11 July 2005}
\conferencename{Groups, homotopy and configuration spaces,
                in honour of Fred Cohen's 60th birthday}
\conferencelocation{University of Tokyo, Japan}

\editor{Norio Iwase}
\givenname{Norio}
\surname{Iwase}

\editor{Toshitake Kohno}
\givenname{Toshitake}
\surname{Kohno}

\editor{Ran Levi}
\givenname{Ran}
\surname{Levi}

\editor{Dai Tamaki}
\givenname{Dai}
\surname{Tamaki}

\editor{Jie Wu}
\givenname{Jie}
\surname{Wu}

\title{Cohomology of Artin  groups of type $\tilde{A}_n$, $B_n$ and
applications}
\author{Filippo Callegaro}
\givenname{Filippo}
\surname{Callegaro}
\address{Scuola Normale Superiore\\Piazza dei Cavalieri, 7\\\newline
56126 Pisa\\Italy}
\email{f.callegaro@sns.it}
\urladdr{}

\author{Davide Moroni}
\givenname{Davide}
\surname{Moroni}
\email{davide.moroni@isti.cnr.it}
\urladdr{}

\author{Mario Salvetti}
\givenname{Mario}
\surname{Salvetti}
\address{Dipartimento di Matematica ``G.Castelnuovo''\\Universit\`a di
Roma ``La Sapienza''\\\newline
Piazza Aldo Moro, 2\\00185 Roma\\Italy\\\newline
and\\\newline
Istituto di Scienza e Tecnologie dell'Informazione ISTI-CNR\\\newline
Via G. Moruzzi 1\\56124 Pisa\\Italy
\vspace{3pt}\\\newline
Dipartimento di Matematica ``L.Tonelli''\\
Universit\`a di Pisa\\\newline
Largo B. Pontecorvo, 5\\56127 Pisa\\Italy}
\email{salvetti@dm.unipi.it}
\urladdr{}

\volumenumber{13}
\issuenumber{}
\publicationyear{2008}
\papernumber{04}
\startpage{85}
\endpage{104}

\doi{}
\MR{}
\Zbl{}


\received{30 May 2006}
\revised{17 January 2007}
\accepted{24 January 2007}
\published{22 February 2008}
\publishedonline{22 February 2008}
\proposed{}
\seconded{}
\corresponding{}
\version{}

\arxivreference{}

\makeatletter
\def\cnewtheorem#1[#2]#3{\newtheorem{#1}{#3}[section]
\expandafter\let\csname c@#1\endcsname\c@teo}

\AtBeginDocument{\let\bar\wbar\let\tilde\wtilde\let\hat\what}
\def\SetFigFont#1#2#3#4#5{\small}
\def\adjustlabel<#1,#2>#3{\smash{\rlap{\kern #1 \raise #2\hbox{#3}}}}

\newtheorem{teo}{Theorem}[section]
\cnewtheorem{prop}[teo]{Proposition}
\cnewtheorem{lem}[teo]{Lemma}
\cnewtheorem{cor}[teo]{Corollary}
\theoremstyle{definition}
\cnewtheorem{df}[teo]{Definition}
\newtheorem*{rem}{Remark}

\makeatother  %  move after \newtheorem block
\makeautorefname{teo}{Theorem}

\def\W{W}

\newcommand{\ph}[0]{\varphi}

\newcommand{\sst}[0]{\subset}

\newcommand{\qbin}[2]{ \left[ \begin{array}{c} #1 \\ #2
    \end{array} \right] }

\renewcommand{\ni}[0]{\noindent}

\newcommand{\tss}[0]{\supset}
\newcommand{\pmu}[0]{{\pm 1}}
\newcommand{\CA}[0]{\overline{C}}

\begin{document}
%
%%%%%%%%%%%%%%%%%%%%%%%%%%%%%%%%%%%%%%%%%%%%%%%%%

\begin{htmlabstract}
<p class="noindent"> We consider two natural embeddings between Artin
groups: the group G<sub>&Atilde;<sub>n-1</sub></sub> of type
&Atilde;<sub>n-1</sub> embeds into the group
G<sub>B<sub>n</sub></sub> of type B<sub>n</sub>;
G<sub>B<sub>n</sub></sub> in turn embeds into the classical braid
group Br<sub>n+1</sub>:=G<sub>A<sub>n</sub></sub> of type
A<sub>n</sub>.  The cohomologies of these groups are related, by
standard results, in a precise way. By using techniques developed in
previous papers, we give precise formulas (sketching the proofs) for
the cohomology of G<sub>B<sub>n</sub></sub> with coefficients over the
module <b>Q</b>[q<sup>&plusmn;1</sup>,t<sup>&plusmn;1</sup>], where
the action is (-q)&ndash;multiplication for the standard generators
associated to the first n-1 nodes of the Dynkin diagram, while is
(-t)&ndash;multiplication for the generator associated to the last
node.  </p> <p class="noindent"> As a corollary we obtain the rational
cohomology for G<sub>&Atilde;<sub>n</sub></sub> as well as the
cohomology of Br<sub>n+1</sub> with coefficients in the
(n+1)&ndash;dimensional representation obtained by Tong, Yang and Ma.
</p> <p class="noindent"> We stress the topological significance,
recalling some constructions of explicit finite CW&ndash;complexes for
orbit spaces of Artin groups. In case of groups of infinite type, we
indicate the (few) variations to be done with respect to the finite
type case. For affine groups, some of these orbit spaces are known to
be K(&pi;,1) spaces (in particular, for type &Atilde;<sub>n</sub>).
</p> <p class="noindent"> We point out that the above cohomology of
G<sub>B<sub>n</sub></sub> gives (as a module over the monodromy
operator) the rational cohomology of the fibre (analog to a Milnor
fibre) of the natural fibration of K(G<sub>B<sub>n</sub></sub>,1) onto
the 2&ndash;torus.  </p>
\end{htmlabstract}

\begin{abstract}
We consider two natural embeddings between Artin groups: the group
$G_{\tilde{A}_{n-1}}$ of type $\tilde{A}_{n-1}$ embeds into the
group  $G_{B_n}$ of type  $B_n$;
 $G_{B_n}$ in turn embeds into the classical braid group
$Br_{n+1}:=G_{A_n}$ of type $A_n$.
The cohomologies of these groups are related, by standard results, in
a precise way. By
using techniques developed in previous papers, we give precise
formulas (sketching the proofs) for  the cohomology of $G_{B_n}$ with
coefficients
over the module $\mathbb{Q}[q^{\pm1},t^{\pm1}],$ where the action is
$(-q)$--multiplication for the standard generators associated to the first
$n-1$ nodes of the Dynkin diagram, while is $(-t)$--multiplication for the
generator associated to the last node.

As a corollary we obtain the rational cohomology for
$G_{\tilde{A}_n}$ as well as the cohomology of $Br_{n+1}$ with
coefficients in the $(n+1)$--dimensional representation obtained by
Tong, Yang and Ma \cite{tong}.

We stress the topological significance, recalling some constructions
of explicit finite CW--complexes for orbit spaces of Artin groups. In
case of groups of infinite type, we indicate the (few) variations to
be done with respect to the finite type case (see Salvetti
\cite{S2}). For affine groups, some of these orbit spaces are known to
be $K(\pi,1)$ spaces (in particular, for type $\tilde{A}_n$).

We point out that the above cohomology of $G_{B_n}$ gives (as a module
over the monodromy operator) the
rational cohomology of the fibre (analog to a Milnor fibre)
of the natural fibration of $K(G_{B_n},1)$ onto the $2$--torus.
\end{abstract}

\begin{webabstract}
We consider two natural embeddings between Artin groups: the group
$G_{\tilde{A}_{n-1}}$ of type $\tilde{A}_{n-1}$ embeds into the group
$G_{B_n}$ of type $B_n$; $G_{B_n}$ in turn embeds into the classical
braid group $Br_{n+1}:=G_{A_n}$ of type $A_n$.  The cohomologies of
these groups are related, by standard results, in a precise way. By
using techniques developed in previous papers, we give precise
formulas (sketching the proofs) for the cohomology of $G_{B_n}$ with
coefficients over the module $\mathbb Q[q^{\pm 1},t^{\pm 1}],$ where
the action is $(-q)$--multiplication for the standard generators
associated to the first $n-1$ nodes of the Dynkin diagram, while is
$(-t)$--multiplication for the generator associated to the last node.

As a corollary we obtain the rational cohomology for $G_{\tilde{A}_n}$
as well as the cohomology of $Br_{n+1}$ with coefficients in the
$(n+1)$--dimensional representation obtained by Tong, Yang and Ma.

We stress the topological significance, recalling some constructions
of explicit finite CW--complexes for orbit spaces of Artin groups. In
case of groups of infinite type, we indicate the (few) variations to
be done with respect to the finite type case. For affine groups, some
of these orbit spaces are known to be $K(\pi,1)$ spaces (in
particular, for type $\tilde{A}_n$).

We point out that the above cohomology of $G_{B_n}$ gives (as a module
over the monodromy operator) the
rational cohomology of the fibre (analog to a Milnor fibre)
of the natural fibration of $K(G_{B_n},1)$ onto the $2$--torus.
\end{webabstract}

\begin{asciiabstract}
We consider two natural embeddings between Artin groups: the group
G_{tilde{A}_{n-1}} of type tilde{A}_{n-1} embeds into the group
1G_{B_n} of type B_n; G_{B_n} in turn embeds into the classical braid
group Br_{n+1}:=G_{A_n} of type A_n.  The cohomologies of these groups
are related, by standard results, in a precise way. By using
techniques developed in previous papers, we give precise formulas
(sketching the proofs) for the cohomology of G_{B_n} with coefficients
over the module Q[q^{+-1},t^{+-1}], where the action is
(-q)-multiplication for the standard generators associated to the
first n-1 nodes of the Dynkin diagram, while is (-t)-multiplication
for the generator associated to the last node.

As a corollary we obtain the rational cohomology for G_{tilde{A}_n}
as well as the cohomology of Br_{n+1} with coefficients in the
(n+1)-dimensional representation obtained by Tong, Yang and Ma.

We stress the topological significance, recalling some constructions
of explicit finite CW-complexes for orbit spaces of Artin groups. In
case of groups of infinite type, we indicate the (few) variations to
be done with respect to the finite type case.  For affine groups, some
of these orbit spaces are known to be K(pi,1) spaces (in particular,
for type tilde{A}_n).

We point out that the above cohomology of G_{B_n} gives (as a module
over the monodromy operator) the rational cohomology of the fibre
(analog to a Milnor fibre) of the natural fibration of K(G_{B_n},1)
onto the 2-torus.
\end{asciiabstract}

\maketitle

%%%%%%%%%%%%%%%%%%%%%%%%%%%%%%%%%%%%%%%%%%%%%%%%%
%%%%%%%%%%%%%%%%%%%%%%%%%%%%%%%%%%%%%%%%%%%%%%%%%
\section{Introduction} The cohomology of classical \emph{braid groups} with
trivial coefficients was computed in the seventies by F Cohen
\cite{coh}, and independently by A Va{\u\i}n{\v{s}}te{\u\i}n \cite{Ve}
(see also Arnol'd \cite{Ar}, Brieskorn and Saito \cite{Br,Br-Sa} and
Fuks \cite{Fu}). For Artin groups of type $C_n,$\ $D_n$ it was
computed by Gorjunov \cite{Go}, and for exceptional cases by Salvetti
\cite{S2} it was given as a $\Z$--module, while the ring structure was
computed by Landi \cite{La}. Other cohomologies with twisted
coefficients were later considered: an interesting case is over the
module of Laurent polynomials $\mathbb Q[q^{\pm 1}],$ which gives the
$\Q$--cohomology of the Milnor fibre of the naturally associated
bundle. For the case of classical braids many people made
computations, independently and using different methods (Frenkel
\cite{Fr}, Markaryan \cite{Ma}, Callegaro and Salvetti \cite{cas} and
De Concini, Procesi and Salvetti \cite{DPS}), while for cases $C_n,\
D_n$ see De Concini, Procesi, Salvetti and Stumbo \cite{DPS2} (here
the authors use the resolution coming from topological considerations
discovered by De Concini and Salvetti \cite{S2,DS}; an equivalent
resolution was independently discovered by using purely algebraic
methods by Squier \cite{Sq}). Over the integral Laurent polynomials
$\mathbb Z[q^{\pm 1}]$ not many computations exist: see D~Cohen and
Suciu $\cite{CS}$ for the exceptional cases and recently Callegaro
\cite{C2} for the case of braid groups, and De Concini, Salvetti and
Stumbo \cite{DSS} for the top cohomologies in all cases.

As regards Artin groups of non-finite type, some computations were
done by Salvetti and Stumbo \cite{SS} and Charney and Davis \cite{CD}.

In this paper we give a complete computation of the above
cohomologies over $\Q[q^{\pm 1}]$
for the Artin groups
$G_{\tilde{A}_n}$ of affine type $\tilde{A}_n.$
By using a natural embedding of $G_{\tilde{A}_{n-1}}$ into the Artin group
$G_{B_n}$ of type $B_n,$ we reduce to the equivalent computation
of the cohomology of $G_{B_n}$ over the module $\Q[q^{\pm 1},t^{\pm 1}],$ where the action is
$(-q)$--multiplication for the standard generators associated to the first
$n-1$ nodes of the Dynkin diagram, while is $(-t)$--multiplication for the
generator associated to the last node.

The proof uses  techniques similar to
\cite{DPS}: a natural
filtration of the complex given in \cite{S2}
and the associated spectral sequence. We sketch the argument here:
details will appear elsewhere.

As a corollary we derive the trivial $\Q$--cohomology of
$G_{\tilde{A}_{n-1}}.$

By using another natural inclusion, a map of $G_{\tilde{B}_{n}}$ into
the classical braid group $Br_{n+1}:=G_{A_n}$, we also find an
isomorphism with the cohomology of ${Br}_{n+1}$ over a certain
interesting representation, namely the irreducible
$(n+1)$--dimensional representation of $Br_{n+1}$ found by Tong, Yang
and Ma \cite{tong}, twisted by an abelian representation.

We also describe the cohomology of the
braid group over the irreducible representation in \cite{tong}.

The topological counterpart is given by some very explicit
constructions of finite CW--complexes which are retracts of the orbit
spaces associated to Artin groups. Following \cite{S2}, we show the
(few) variations to be done in case of groups of infinite type,
explicitly showing the affine case (see also \cite{CD} for a different
construction). From such constructions the standard presentation of
the fundamental group comes quite easily (see D{\~u}ng \cite{Ng}). We
also easily deduce a formula for the Euler characteristic of the orbit
space in the affine case. It is conjectured that such orbit spaces are
always $K(\pi,1)$ spaces; for the affine groups, this is known in case
$\tilde{A}_n,\ \tilde{C}_n$ (see Okonek \cite{Ok} and
Charney and Peifer \cite{CP}; see also \cite{CD} for
a different class of Artin groups of infinite type).

It is interesting to notice the geometrical meaning of the
two-parameters cohomology of $G_{{B}_{n}}:$  similar to the
one-parameter case, it gives the trivial cohomology of the
``Milnor fibre'' associated to the natural map of the orbit space
onto a two-dimensional torus.

The second author was partially supported by ISTI-CNR.  The third
author was partially supported (40\%) by M.U.R.S.T.
%%%%%%%%%%%%%%%%%%%%%%%%%%%%%%%%%%%%%%%%%%%%%%%%%
\section{Preliminary results}
%%%%%%%%%%%%%%%%%%%%%%%%%%%%%%%%%%%%%%%%%%%%%%%%%
In this section we briefly fix the notation and recall some
preliminary results.
%%%%%%%%%%%%%%%%%%%%%%%%%%%%%%%%%%%%%%%%%%%%%%%%%
\subsection{Coxeter groups and Artin  groups}
%%%%%%%%%%%%%%%%%%%%%%%%%%%%%%%%%%%%%%%%%%%%%%%%%
A \emph{Coxeter graph} is a finite undirected graph, whose edges
are labelled with integers $\geq 3$ or with the symbol $\infty$.

Let  $S$ be the vertex set of a Coxeter graph. For every pair of
vertices $s, t \in S$ ($s\neq t$) joined by an edge, define
$m(s,t)$ to be the label of the edge joining them. If $s, \, t$
are not joined by an edge, set by convention $m(s,t)=2$. Let also
$m(s,s)=1$ (see Bourbaki \cite{bour} and Humphreys \cite{hum}).

Two groups are associated to  a Coxeter graph: the \emph{Coxeter
group} $W$ defined by
$$
W=\langle s \in S \, |\, (st)^{m(s,t)}=1 \,\, \forall s,t \in S
\,\, \textrm{such that}\,\, m(s,t)\neq \infty\rangle
$$
and the \emph{Artin group} $G$ defined by (see Brieskorn and Saito 
\cite{Br-Sa} and Deligne \cite{del}):
$$
G=\langle s \in S \, | \,
\begin{underbrace}{stst\ldots}\end{underbrace}_{m(s,t)-\mathrm{terms}}=\begin{underbrace}{tsts\ldots}\end{underbrace}_{m(s,t)-\mathrm{terms}}
\,\, \forall s,t \in S \,\, \textrm{such that}\,\, m(s,t)\neq
\infty\rangle.
$$
Loosely speaking, $G$ is the group obtained by dropping the
relations $s^2=1$ ($s\in S$) in the presentation for $W$.

In this paper, we are primarily interested in Artin  groups
associated to Coxeter graphs of type $A_n$, $B_n$ and
$\tilde{A}_{n-1}$ (see \fullref{fig:dinkyn}).
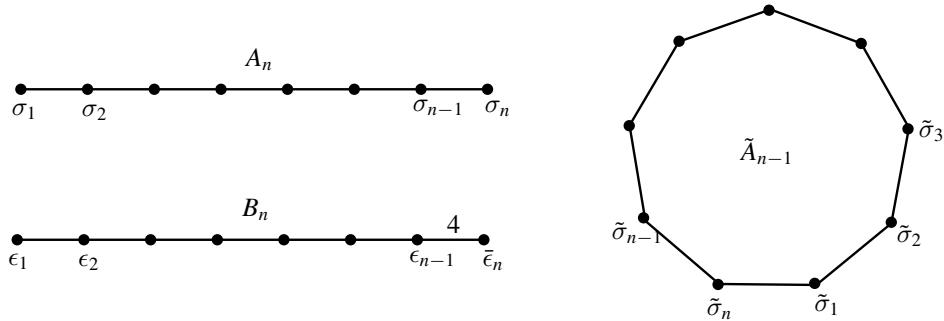
\begin{figure}[ht!]
\begin{center}
\begin{picture}(0,0)%
\includegraphics{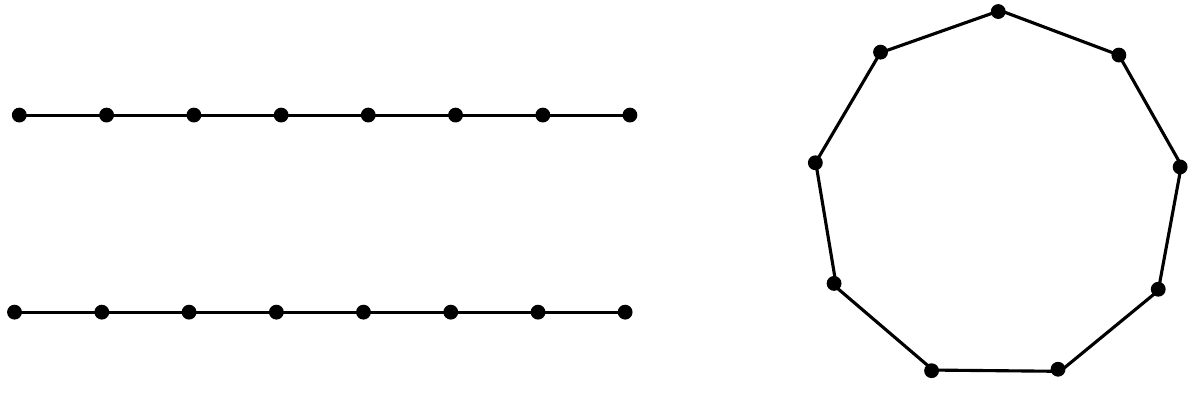}%
\end{picture}%
\setlength{\unitlength}{3947sp}%
\begin{picture}(5720,1937)(-4070,-4328)
\put(-1076,-3076){\makebox(0,0)[lb]{\smash{{\SetFigFont{9}{10.8}{\familydefault}{\mddefault}{\updefault}{\color[rgb]{0,0,0}$\sigma_n$}%
}}}}
\put(1510,-3891){\makebox(0,0)[lb]{\smash{{\SetFigFont{9}{10.8}{\familydefault}{\mddefault}{\updefault}{\color[rgb]{0,0,0}$\tilde{\sigma}_2$}%
}}}}
\put(1650,-3228){\makebox(0,0)[lb]{\smash{{\SetFigFont{9}{10.8}{\familydefault}{\mddefault}{\updefault}{\color[rgb]{0,0,0}$\tilde{\sigma}_3$}%
}}}}
\put(994,-4289){\makebox(0,0)[lb]{\smash{{\SetFigFont{9}{10.8}{\familydefault}{\mddefault}{\updefault}\adjustlabel<0pt,-2pt> {\color[rgb]{0,0,0}$\tilde{\sigma}_1$}%
}}}}
\put(534,-3376){\makebox(0,0)[lb]{\smash{{\SetFigFont{12}{14.4}{\familydefault}{\mddefault}{\updefault}{\color[rgb]{0,0,0}$\tilde{A}_{n-1}$}%
}}}}
\put(-4070,-4036){\makebox(0,0)[lb]{\smash{{\SetFigFont{9}{10.8}{\familydefault}{\mddefault}{\updefault}{\color[rgb]{0,0,0}$\epsilon_1$}%
}}}}
\put(-3630,-4043){\makebox(0,0)[lb]{\smash{{\SetFigFont{9}{10.8}{\familydefault}{\mddefault}{\updefault}{\color[rgb]{0,0,0}$\epsilon_2$}%
}}}}
\put(-1313,-3834){\makebox(0,0)[lb]{\smash{{\SetFigFont{9}{10.8}{\familydefault}{\mddefault}{\updefault}{\color[rgb]{0,0,0}$4$}%
}}}}
\put(-2598,-3729){\makebox(0,0)[lb]{\smash{{\SetFigFont{12}{14.4}{\familydefault}{\mddefault}{\updefault}{\color[rgb]{0,0,0}${B}_{n}$}%
}}}}
\put(-1552,-4015){\makebox(0,0)[lb]{\smash{{\SetFigFont{9}{10.8}{\familydefault}{\mddefault}{\updefault}{\color[rgb]{0,0,0}$\epsilon_{n-1}$}%
}}}}
\put(-4047,-3090){\makebox(0,0)[lb]{\smash{{\SetFigFont{9}{10.8}{\familydefault}{\mddefault}{\updefault}{\color[rgb]{0,0,0}$\sigma_1$}%
}}}}
\put(-3607,-3097){\makebox(0,0)[lb]{\smash{{\SetFigFont{9}{10.8}{\familydefault}{\mddefault}{\updefault}{\color[rgb]{0,0,0}$\sigma_2$}%
}}}}
\put(-1529,-3069){\makebox(0,0)[lb]{\smash{{\SetFigFont{9}{10.8}{\familydefault}{\mddefault}{\updefault}{\color[rgb]{0,0,0}$\sigma_{n-1}$}%
}}}}
\put(-2575,-2783){\makebox(0,0)[lb]{\smash{{\SetFigFont{12}{14.4}{\familydefault}{\mddefault}{\updefault}{\color[rgb]{0,0,0}${A}_{n}$}%
}}}}
\put(-278,-3870){\makebox(0,0)[lb]{\smash{{\SetFigFont{9}{10.8}{\familydefault}{\mddefault}{\updefault}{\color[rgb]{0,0,0}$\tilde{\sigma}_{n-1}$}%
}}}}
\put(317,-4297){\makebox(0,0)[lb]{\smash{{\SetFigFont{9}{10.8}{\familydefault}{\mddefault}{\updefault}{\color[rgb]{0,0,0}\adjustlabel<0pt,-3pt> {$\tilde{\sigma}_n$}}%
}}}}
\put(-1092,-4022){\makebox(0,0)[lb]{\smash{{\SetFigFont{9}{10.8}{\familydefault}{\mddefault}{\updefault}\adjustlabel<0pt,-1pt> {\color[rgb]{0,0,0}$\bar{\epsilon}_n$}%
}}}}
\end{picture}%
%
%%%%%%%%%%%%%%%%%%%%%%%%%%%%%%%%%%%%%%%%%%%%%%%%%
\end{center}
\caption{Coxeter graph of type  $A_n$, $B_n$ ($n\geq 2$) and
$\tilde{A}_{n-1}$ ($n\geq 3$). Labels equal to $3$, as usual, are
not shown. Moreover, to fix notation, every vertex is labelled
with the corresponding generator in the Artin 
group.}
\label{fig:dinkyn}
\end{figure}
%%%%%%%%%%%%%%%%%%%%%%%%%%%%%%%%%%%%%%%%%%%%%%%%%
%
%%%%%%%%%%%%%%%%%%%%%%%%%%%%%%%%%%%%%%%%%%%%%%%%%
\subsection{Inclusions of Artin groups}
%%%%%%%%%%%%%%%%%%%%%%%%%%%%%%%%%%%%%%%%%%%%%%%%%
%
Let $\mathrm{Br}_{n+1}:=G_{A_n}$ be the braid group on $n+1$
strands and $\mathrm{Br}_{n+1}^{n+1}<\mathrm{Br}_{n+1}$ be the
subgroup of braids fixing the $(n+1)$-st strand.
%$(n+1)$-pure braids (ie, the subgroup of braids permuting only the first $n$ strands).
The group
$\mathrm{Br}_{n+1}^{n+1}$ is called the annular braid group, since
it can be regarded as the group of braids on $n$ strands on the
annulus (see \fullref{fig:cylinder}).
\begin{figure}[ht!]
 \begin{center}
\begin{picture}(0,0)%
\includegraphics{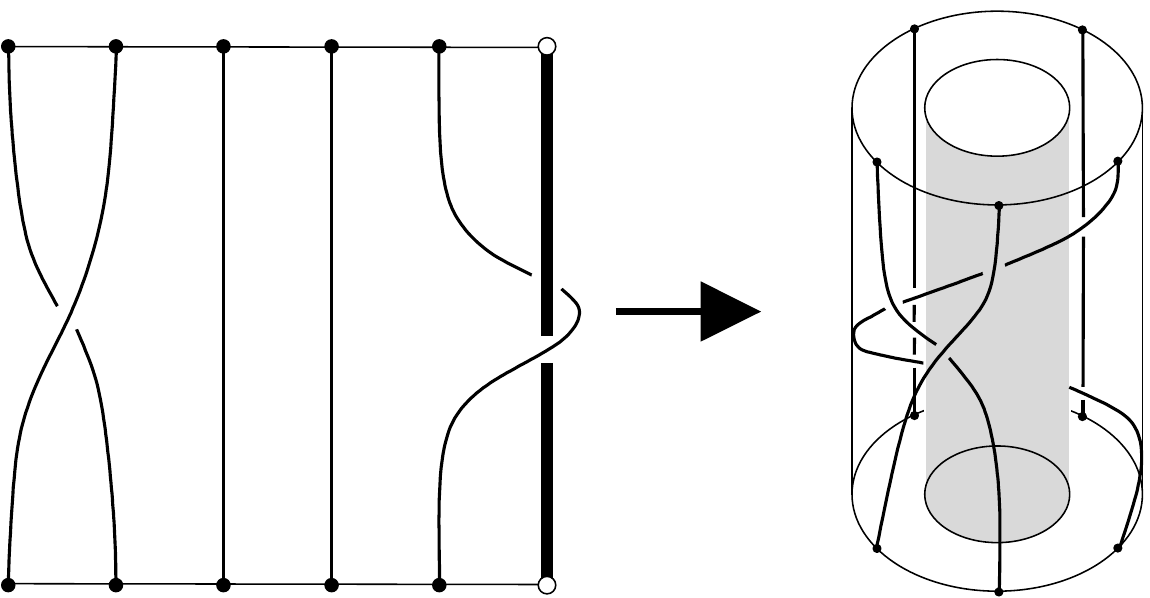}%
\end{picture}%
\setlength{\unitlength}{3947sp}%
\begin{picture}(5501,2851)(-2605,-4336)
\put(2588,-1561){\makebox(0,0)[lb]{\smash{\SetFigFont{8}{9.6}{\familydefault}{\mddefault}{\updefault}{\color[rgb]{0,0,0}$4$}%
}}}
\put(1708,-1561){\makebox(0,0)[lb]{\smash{\SetFigFont{8}{9.6}{\familydefault}{\mddefault}{\updefault}{\color[rgb]{0,0,0}$3$}%
}}}
\put(2598,-3602){\makebox(0,0)[lb]{\smash{\SetFigFont{8}{9.6}{\familydefault}{\mddefault}{\updefault}\adjustlabel<-2pt,-2pt> {\color[rgb]{0,0,0}$4$}%
}}}
\put(1728,-3591){\makebox(0,0)[lb]{\smash{\SetFigFont{8}{9.6}{\familydefault}{\mddefault}{\updefault}\adjustlabel<0pt,-2pt> {\color[rgb]{0,0,0}$3$}%
}}}
\put(-1566,-1635){\makebox(0,0)[lb]{\smash{\SetFigFont{8}{9.6}{\familydefault}{\mddefault}{\updefault}\adjustlabel<-2pt,0pt> {\color[rgb]{0,0,0}$3$}%
}}}
\put(-1041,-1636){\makebox(0,0)[lb]{\smash{\SetFigFont{8}{9.6}{\familydefault}{\mddefault}{\updefault}\adjustlabel<-2pt,0pt> {\color[rgb]{0,0,0}$4$}%
}}}
\put(1640,-2238){\makebox(0,0)[lb]{\smash{\SetFigFont{8}{9.6}{\familydefault}{\mddefault}{\updefault}{\color[rgb]{0,0,0}$2$}%
}}}
\put(2655,-2238){\makebox(0,0)[lb]{\smash{\SetFigFont{8}{9.6}{\familydefault}{\mddefault}{\updefault}{\color[rgb]{0,0,0}$5$}%
}}}
\put(2200,-2428){\makebox(0,0)[lb]{\smash{\SetFigFont{8}{9.6}{\familydefault}{\mddefault}{\updefault}{\color[rgb]{0,0,0}$1$}%
}}}
\put(1634,-4113){\makebox(0,0)[lb]{\smash{\SetFigFont{8}{9.6}{\familydefault}{\mddefault}{\updefault}{\color[rgb]{0,0,0}$2$}%
}}}
\put(2200,-4284){\makebox(0,0)[lb]{\smash{\SetFigFont{8}{9.6}{\familydefault}{\mddefault}{\updefault}{\color[rgb]{0,0,0}$1$}%
}}}
\put(2639,-4116){\makebox(0,0)[lb]{\smash{\SetFigFont{8}{9.6}{\familydefault}{\mddefault}{\updefault}{\color[rgb]{0,0,0}$5$}%
}}}
\put(-2583,-1637){\makebox(0,0)[lb]{\smash{\SetFigFont{8}{9.6}{\familydefault}{\mddefault}{\updefault}\adjustlabel<-2pt,0pt> {\color[rgb]{0,0,0}$1$}%
}}}
\put(-2077,-1640){\makebox(0,0)[lb]{\smash{\SetFigFont{8}{9.6}{\familydefault}{\mddefault}{\updefault}\adjustlabel<-2pt,0pt> {\color[rgb]{0,0,0}$2$}%
}}}
\put(-530,-1633){\makebox(0,0)[lb]{\smash{\SetFigFont{8}{9.6}{\familydefault}{\mddefault}{\updefault}\adjustlabel<-2pt,0pt> {\color[rgb]{0,0,0}$5$}%
}}}
\put(-10,-1628){\makebox(0,0)[lb]{\smash{\SetFigFont{8}{9.6}{\familydefault}{\mddefault}{\updefault}\adjustlabel<-2pt,0pt> {\color[rgb]{0,0,0}$6$}%
}}}
\end{picture}
%%%%%%%%%%%%%%%%%%%%%%%%%%%%%%%%%%%%%%%%%%%%%%%%%
\end{center}
\caption{A braid in $\mathrm{Br}_{6}^{6}$ represented as an
annular braid on $5$ strands.}
\label{fig:cylinder}
\end{figure}
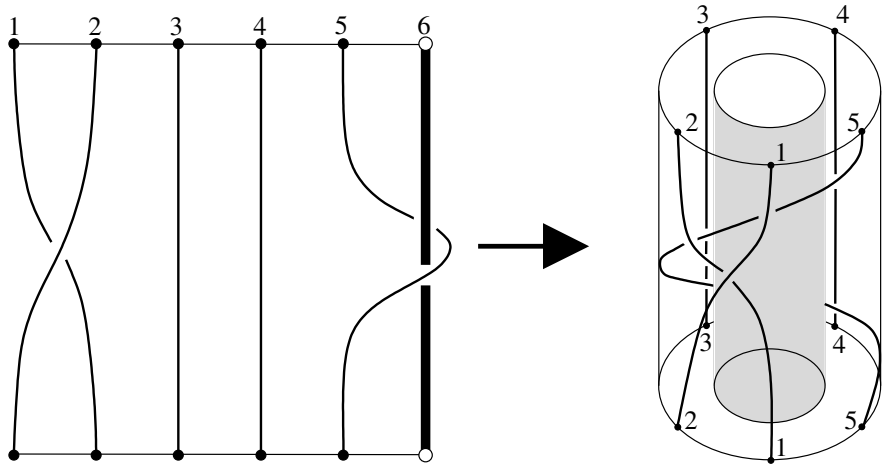
%%%%%%%%%%%%%%%%%%%%%%%%%%%%%%%%%%%%%%%%%%%%%%%%%

It is well known that the annular braid group is indeed isomorphic to
the Artin group $G_{B_n}$ of type $B_n$.  For a proof of the following
Theorem see Lambropoulou \cite{lam} or Crisp \cite{crisp}.
\begin{teo}
Let $\sigma_1,
\ldots, \sigma_n$ be the standard generators for $G_{A_n}$ and
let $\epsilon_1, \ldots,$
$\epsilon_{n-1}, \bar{\epsilon}_n$ be the generators for
$G_{B_n}$.
The map
\begin{align*}
G_{B_n}&\to \mathrm{Br}_{n+1}^{n+1}<\mathrm{Br}_{n+1}\\
\epsilon_i& \mapsto  \sigma_i \quad \textrm{for $ 1 \leq i \leq n-1$}\\
\bar{\epsilon}_n &\mapsto \sigma_n^2
\end{align*}
 is an isomorphism.
\end{teo}
 Using the suggestion given by the identification with the
annular braid group, a new interesting presentation for $G_{B_n}$
can be worked out.
Let $\tau=\bar{\epsilon}_n \epsilon_{n-1} \cdots \epsilon_2
\epsilon_1$. See \fullref{fig:tau_generator}. 
\begin{figure}[ht!]
 \begin{center}
\includegraphics[scale=.7]{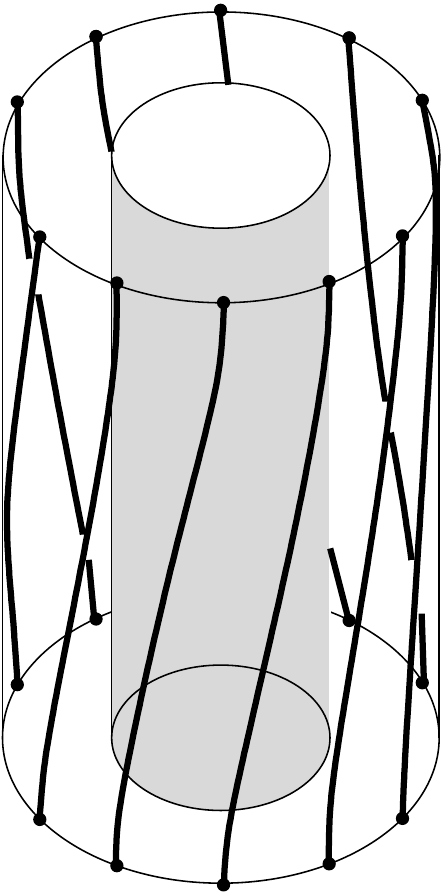}%
\end{center}
\caption{As an annular braid the element $\tau$ is obtained
turning the bottom annulus by a rotation of $2 \pi /n$.}
\label{fig:tau_generator}
\end{figure}
%%%%%%%%%%%%%%%%%%%%%%%%%%%%%%%%%%%%%%%%%%%%%%%%%

It is easy to verify that:
$$
\tau^{-1} \epsilon_i \tau= \epsilon_{i+1} \quad \textrm{ for $1
\leq i < n-1$}
$$
ie, conjugation by $\tau$ shifts forward the first $n-2$ standard
generators. By analogy, let  $\epsilon_n=\tau^{-1} \epsilon_{n-1}
\tau$.

We have  the following theorem: 
\begin{teo}[Kent and Peifer \cite{peifer}] \label{teo:peifer}
The group $G_{B_n}$ has presentation $\langle \mathcal{G} |
\mathcal{R} \rangle$ where
\begin{align*}
\mathcal{G}=& \{ \tau, \epsilon_1, \epsilon_2, \ldots ,\epsilon_{n} \}\\
\mathcal{R}=& \{ \epsilon_i \epsilon_j=\epsilon_j \epsilon_i \quad
\textrm{for $i\neq j-1,j+1$}\} \cup\\
 {}& \{\epsilon_i \epsilon_{i+1} \epsilon_i=
\epsilon_{i+1} \epsilon_{i}
\epsilon_{i+1}\}\cup \\
{}& \{\tau^{-1} \epsilon_i \tau= \epsilon_{i+1}  \}
\end{align*}
where are all indexes have to be taken modulo $n$. 
\end{teo}
Letting $\tilde{\sigma}_1, \tilde{\sigma}_2, \ldots,
\tilde{\sigma}_{n}$ be the standard generators of the Artin 
group of type $\tilde{A}_{n-1}$, we have the following immediate
corollary:
\begin{cor}[Kent and Peifer \cite{peifer}, see also tom Dieck
    \cite{tdieck} and Allcock \cite{all}]\label{cor:peifer} The map
$$G_{\tilde{A}_{n-1}} \owns \tilde{\sigma}_i \mapsto \epsilon_i\in G_{B_n}$$
gives an isomorphism between $G_{\tilde{A}_{n-1}}$ and the
subgroup of $G_{B_n}$ generated by\break $\epsilon_1, \ldots,
\epsilon_n$. Moreover, we have a semidirect product decomposition
$G_{B_n}\cong G_{\tilde{A}_{n-1}} \rtimes \langle \tau
\rangle$.
\end{cor}

We have thus a `curious' inclusion of the Artin group of infinite
type $\tilde{A}_{n-1}$ into the Artin group of finite type
${B_n}$.

\begin{rem}
The proof of \fullref{teo:peifer} presented in \cite{peifer}
is algebraic and based on Tietze moves; a somewhat more
coincise proof is obtained by standard topological
constructions. Indeed, one can exhibit an explicit infinite cyclic
covering $K(G_{\tilde{A}_{n-1}},1)\to K(G_{B_{n}},1)$ (see \cite{all}).
\end{rem}

%%%%%%%%%%%%%%%%%%%%%%%%%%%%%%%%%%%%%%%%%%%%%%%%%
\subsection{Local coefficients and induced representations}
%%%%%%%%%%%%%%%%%%%%%%%%%%%%%%%%%%%%%%%%%%%%%%%%%

The previous inclusions allow to relate the homology of the
involved groups by means of Shapiro's lemma (see for instance
Brown \cite{brown}), of which we explore some consequences.

Let $M:=\mathbb{Q}[q^{\pm 1}].$  We indicate by $M_q$ the
$G_{\tilde{A}_{n-1}}$--module
where the action of the standard generators is $(-q)$--multiplication.

\begin{prop}\label{prop:sha1}
We have
$$
H_*(G_{\tilde{A}_{n-1}}, M_q)\cong H_*(G_{B_{n}}, M[t^{\pm 1}]_{q,t})
$$
$$
H^*(G_{\tilde{A}_{n-1}}, M_q)\cong H^*(G_{B_{n}}, M[[t^{\pm 1}]]_{q,t})
$$
where the action of $G_{B_n}$ on $M[t^{\pm 1}]_{q,t}$ (and on
$M[[t^{\pm 1}]]_{q,t}$) is given by $(-q)$--multiplication for the
generators $\epsilon_1, \ldots, \epsilon_{n-1}$ and
$(-t)$--multiplication for the last generator $\bar{\epsilon}_n$.
\end{prop}
\begin{proof}
Applying Shapiro's lemma to the inclusion
$\tilde{A}_{n-1}<G_{B_{n}}$, one obtains:
$$
H_*(G_{\tilde{A}_{n-1}}, M_q)\cong H_*(G_{B_{n}},
\mathrm{Ind}_{G_{\tilde{A}_{n-1}}}^{G_{B_n}}M_q)
$$
$$
H^*(G_{\tilde{A}_{n-1}}, M_q)\cong H^*(G_{B_{n}},
\mathrm{Coind}_{G_{\tilde{A}_{n-1}}}^{G_{B_{n}}}M_q).
$$
By \fullref{cor:peifer}, any element of
$\mathrm{Ind}_{G_{\tilde{A}_{n-1}}}^{G_{B_n}}M_q:=\mathbb{Z}[G_{B_n}]
\otimes_{G_{\tilde{A}_{n-1}}} M_q$ can be represented as a sum of
elements of the form $\tau^{\alpha} \otimes q^m$. Now, we have an
isomorphism of $\mathbb{Z}[G_{B_n}]$--modules
\begin{align*}
\mathbb{Z}[G_{B_n}] \otimes_{G_{\tilde{A}_{n-1}}} M_q &\to
M[t^{\pm 1}]_{q,t}
\end{align*}
defined by sending $ \tau^\alpha \otimes q^m\mapsto (-1)^{n\alpha}
t^{\alpha}q^{(n-1)\alpha +m}$ and the result follows.

In cohomology we have similarly:
$$\mathrm{Coind}_{G_{\tilde{A}_{n-1}}}^{G_{B_n}}M_q:=\mathrm{Hom}_{G_{\tilde{A}_{n-1}}}(\mathbb{Z}[G_{B_n}],M_q)
 \cong M[[t^{\pm 1}]]_{q,t}.\proved$$
\end{proof}

It is interesting to note what happens inducing again via the
inclusion $G_{B_n}<G_{A_n}$.

Let $V=\bigoplus_{i=1}^{n+1} \mathbb{Q}[u^\pmu]e_i$ be an
$(n+1)$--dimensional free $\mathbb{Q}[u^\pmu]$ module.
\begin{df}\label{df:tong}
The Tong--Yang--Ma representation \cite{tong} is the representation
$$\rho:G_{A_n} \to  \mathrm{Gl}_{\mathbb{Q}[u^\pmu]}(V) $$
defined w.r.t. the basis $e_1,  \ldots, e_{n+1}$ by:
$$\rho (\sigma_i)=\left( \begin{array}{ccccc}
I_{i-1}&&&\\
&0&1&\\
&u&0&\\
&&&I_{n-i}
\end{array} \right)$$
where $I_j$ denote the $j$--dimensional identity matrix and all
other entries are zero.
\end{df}

We refer to Sysoeva \cite{sysoeva} for a discussion of the
relevance of Tong--Yang--Ma representation in braid group
representation theory. We recall that the image of the
pure braid group in the Tong--Yang--Ma representation is abelian;
hence this representation factors through the \emph{extended Coxeter group}
presented by Tits \cite{tits}.

\begin{prop}\label{prop:sha2}
We have
$$ H_*(G_{B_n}, M[t^{\pm 1}]_{q,t})\cong H_*(G_{A_n}, M_q \otimes
V)
$$
$$
H^*(G_{B_n}, M[t^{\pm 1}]_{q,t})\cong H^*(G_{A_n}, M_q \otimes V)
$$
where the action of $G_{A_n}$ on $M_q$ is defined sending the
standard generators to $(-q)$--multiplication.
\end{prop}
\begin{proof}[Sketch of proof]
For the statement in homology, by Shapiro's lemma, it is enough to
show that ${\mathrm{Ind}}_{G_{B_n}}^{G_{A_n}} M[t^{\pm 1 }]_{q,t}
\cong M_q \otimes
V$.
Note that $[G_{A_n}:G_{B_n}]=n+1$ and let choose as coset
 representatives for $G_{A_n}/G_{B_n}$ the
 elements $$\alpha_i= (\sigma_i \sigma_{i+1} \cdots \sigma_{n-1})
 \sigma_n (\sigma_{i} \sigma_{i+1} \cdots
 \sigma_{n-1})^{-1}$$for $1 \leq i \leq n-1$, $\alpha_n=\sigma_n$,
 $\alpha_{n+1}=e$. 
 
Then, by definition of induced representation,
 $$\mathrm{Ind}_{G_{B_n}}^{G_{A_n}} M[t^{\pm 1
 }]_{q,t}=\bigoplus_{i=1}^{n+1} M[t^{\pm 1}]e_i$$ with the following action.
 For an element  $x\in G_{A_n}$, write $x \alpha_k= \alpha_{k'} x'$
 with $x' \in G_{B_n}$.
 Then $x$ acts on an element $r\cdot e_k \in \bigoplus_{i=1}^{n+1}
 M[t^{\pm 1}]e_i$
 as $x(r\cdot e_k)=(x'r) \cdot e_{k'}$.

 After some easy computations, one can write the representation in
 the following matrix form:
$$ \sigma_i\mapsto \left( \begin{array}{ccccc}
-qI_{i-1}&&&\\
&0&-q&\\
&q^{-1} t&0&\\
&&&-qI_{n-i}
\end{array} \right)$$
for $1\leq i \leq n-1$, whereas
$$ \sigma_n\mapsto \left( \begin{array}{ccccc}
-qI_{n-1}&&\\
&0&1&\\
& -t&0
\end{array} \right).$$
Conjugating by $U=\mathrm{Diag}(1,1, \ldots, 1,-q^{-1})$ and
setting $u=-q^{-2}t$, one obtains the desired result.

Finally, since $[G_{A_n}:G_{B_n}]=n+1<\infty$, the induced and
coinduced representation are isomorphic; so the analogous
statement in cohomology holds. 
\end{proof}

\begin{rem}
Specializing $q$ to $1$, we have in particular that homology of
$G_{\tilde{A}_{n-1}}$ with trivial coefficients is isomorphic to
homology of $G_{A_n}$ with coefficients in the Tong--Yang--Ma
representation.
\end{rem}

By means of Propositions \ref{prop:sha1} and \ref{prop:sha2}, in
the following all cohomology computations will be performed in
$G_{B_n}$ using the double-weight coefficient system. We conclude
remarking that cohomology computations can be even reduced from
the ring of Laurent series to the ring of Laurent polynomials by
the following result about degree shift, obtained by Callegaro \cite{C}
in a slightly weaker form, but which is possible to extend to our case
with little effort.

\begin{prop}[Degree shift]\label{prop:shift}
$$
H^*(G_{B_n}, M[[t^{\pm 1 }]]_{q,t})\cong H^{*+1}(G_{B_n}, M[t^{\pm
1 }]_{q,t}).
$$

\end{prop}

%%%%%%%%%%%%%%%%%%%%%%%%%%%%%%%%%%%%%%%%%%%%%%%%%
\subsection{${\mathbf{(q,t)}}$--weighted Poincar\'e series for ${\mathbf{B_n}}$}
%%%%%%%%%%%%%%%%%%%%%%%%%%%%%%%%%%%%%%%%%%%%%%%%%

For future use in cohomology computations, we are interested in a
$(q,t)$--analog of the usual Poincar\'e series for $B_n$. This result
and similar ones are studied in Reiner \cite{reiner}, to which we
refer for details. We also use classical results from \cite{bour,hum}
without further reference.

Consider the Coxeter group $W$ of type $B_n$ with its standard
generating
reflections $s_1, s_2, \ldots, s_n$.

For $w \in W$, let $n(w)$ be the number of times $s_n$ appears in
a reduced expression for $w$. By standard facts, $n(w)$ is
well-defined.

 Let also $\W(q,t)=\sum_{w \in W}
q^{\ell(w)-n(w)} t^{n(w)}$ be the $(q,t)$--weighted Poincar\'e
series, where $\ell$ is the length function.

We recall some notation. We write $\ph_m(q)$ for the $m$-th cyclotomic polynomial
in the variable $q$ and we define the $q$--analog of the number $m$ by
the polynomial $$[m ]_q :=
1 + q + \cdots q^{m-1} = \frac{q^m -1}{q-1}.$$ It is easy to see that $[m ] =
\prod_{i \mid m, i \neq 1} \ph_i(q)$. Moreover we define the $q$--factorial
$[m]_q!$ as the product
$$\prod_{i=1}^{m}[i]_q$$ and the $q$--analog of the binomial $\binom{m}{i}$
as the polynomial $$\qbin{m}{i}_{q}: = \frac{[m]_q!}{[i]_q! [m-i]_q!}.$$
We can also define the $(q,t)$--analog of an even number
$$[2m]_{q,t} := [m]_q (1+tq^{m-1})$$
and of the double factorial
$$
[2m]_{q,t}!! := \prod_{i=1}^m [2i]_{q,t}\ =\ [ m ]_q!
\prod_{i=0}^{m-1} (1+tq^i).
$$
Finally, we define the polynomial
\begin{equation}\label{e:binqt}
\qbin{m}{i}_{q,t}': = \frac{[2m]_{q,t}!!}{[2i]_{q,t}!! [m-i]_q!} \ =
\qbin{m}{i}_{q}\prod_{j=i}^{m-1}(1+tq^j).
\end{equation}

\begin{prop}{\rm\cite{reiner}}\label{p:qtpoincare}\qua
$$\W(q,t)=
%[n]_q! \prod_{i=0}^{n-1}(1+q^it)=
[2n]_{q,t}!!.$$
\end{prop}
\begin{proof} Consider the parabolic subgroup $W_I$ associated to the subset of
reflections $I=\{s_1, \ldots, s_{n-1}\}$. Notice that $W_I$ is
isomorphic to the symmetric group on $n$ letters $A_{n-1}$ and
that it has index $2^n$ in $B_n$. Let $W^I$ be the set of minimal
coset representatives for $W/W_I$. Then, by multiplicative
properties on reduced expressions:
\begin{align}
\W(q,t)=&\sum_{w \in W} q^{\ell(w)-n(w)} t^{n(w)}\nonumber\\
=& \Big( \sum_{w' \in W^I} q^{\ell(w')-n(w')} t^{n(w')}\Big) \cdot
\Big( \sum_{w'' \in W_I} q^{\ell(w'')-n(w'')} t^{n(w'')}\Big).
\label{formula:1}
\end{align}
Clearly, for elements $w''\in W_I$, we have $n(w'')=0$; so the
second factor in \eqref{formula:1} reduces to the well-known
Poincar\'e series for $A_{n-1}$:
$$
 \sum_{w'' \in W_I} q^{\ell(w'')-n(w'')} t^{n(w'')}=[n]_q!.
$$
The remaining of the proof uses an explicit computation of minimal
coset representatives in $W^I.$
\end{proof}

%%%%%%%%%%%%%%%%%%%%%%%%%%%%%%%%%%%%%%%%%%%%%%%%%
\section{The cohomology of $G_{B_n}$}\label{Section3}
%%%%%%%%%%%%%%%%%%%%%%%%%%%%%%%%%%%%%%%%%%%%%%%%%

In this section we will compute the cohomology groups
$H^*(G_{B_n}, R_{q,t})$, where $R_{q,t}$ is the local system over
the ring of Laurent polynomials $R = \Q [q^\pmu, t^\pmu]$ and the
action is $(-q)$--multiplication for the standard generators
associated to the first $n-1$ nodes of the Dynkin diagram, while
is $(-t)$--multiplication for the generator associated to the last
node.

In order to state our result we need to define,
for $m\geq 2,$  $R$--modules
$$
\{m \}_i = R/(\ph_m(q), q^it+1).
$$
For $m = 1$ we set:
$$
\{ 1 \}_i = R/(q^it +1).
$$
Notice that the modules $\{ m \}_i$ are pairwise non isomorphic as
$R$--modules. $\{m \}_i$ and
$\{m'\}_{i'}$ are isomorphic as $\Q[q^\pmu]$--modules if and only if $m
= m'$ and are isomorphic as
$\Q[t^\pmu]$--modules if and only if $\phi(m) = \phi(m')$ and
$\frac{m}{(m,i)} = \frac{m'}{(m,i')}$.

Our main result is the following:

\begin{teo} \label{t:cohomqt}
$$
H^i(G_{B_n}, R_{q,t})= \left\{
\begin{array}{ll}
\bigoplus_{d \mid n, 0 \leq k \leq d-2}
\{ d \}_k \oplus \{ 1 \}_{n-1} & \mbox{ if } i = n\\
\bigoplus_{
d \mid n, 0 \leq k \leq d-2, d \leq \frac{n}{j+1} }
\{ d \}_k & \mbox{ if }i = n -2j \\
\bigoplus_{d \nmid n,
d \leq \frac{n}{j+1} }
\{ d \}_{n-1} & \mbox{ if }i = n -2j -1
\end{array}
\right.
$$

\end{teo}

To perform our computation we will use a method quite similar to
\cite{DPS}, namely the complex
introduced in \cite{S2}, and the spectral sequence induced by a natural
filtration.

Recall from \cite{S2} that the complex that compute the cohomology
of $G_{B_n}$ over $R_{q,t}$
is given as follows:
$$
C_n^* = \bigoplus_{\Gamma \sst I_n} R.\Gamma
$$
where $I_n$ denote the set $\{1, \ldots, n\}$ and the graduation is given
by $\mid \Gamma \mid$.

The set $I_n$ corresponds to the set of nodes of the Dynkin diagram of
$B_n$ and
in particular the last element, $n$, corresponds to the last node.

It is useful to consider also the analog complex $\CA_n^*$ for the cohomology
of $G_{A_n}$
on the local system $R_{q,t}$. In this case the action associated to a
standard generator is always the $(-q)$--multiplication and so the
complex $\CA_n^*$ and its cohomology are free as $\Q[t^\pm]$--modules. The
complex $\CA_n^*$ is isomorphic to $C_n^*$ as an $R$--module.
In both complexes the coboundary map is
\begin{equation} \label{e:bordo}
\delta (q,t) (\Gamma) = \sum_{j \in I_n \setminus \Gamma}
(-1)^{\sigma(j,\Gamma)}
\frac{W_{\Gamma \cup \{ j \}}(q,t)} {W_{ \Gamma}(q,t)} (\Gamma \cup \{j\})
\end{equation}
\ni where $\sigma(j, \Gamma)$ is the number of elements of
$\Gamma$ that are less than $j$ in the natural ordering. In the
case $A_n$,  $ W_\Gamma (q,t)$ is the Poincar\'e polynomial of the
parabolic subgroup $W_\Gamma \sst A_n$ generated by the elements
in the set $\Gamma$, with weight $q$ for each standard generator,
while in the case $B_n$ $ W_\Gamma (q,t)$ is the Poincar\'e
polynomial of the parabolic subgroup $W_\Gamma \sst B_n$ generated
by the elements in the set $\Gamma$, with weight $q$ for the first
$n-1$ generators and $t$ for the last generator.

Using \fullref{p:qtpoincare} we can give an explicit
computation of the coefficients appearing in \ref{e:bordo}. For
any $ \Gamma \sst I_n$, let $\overline{\Gamma}$ be the subgraph of
the Dynkin diagram $B_n$ which is spanned by $\Gamma$. Recall that
if $\overline{\Gamma}$ is a connected component of the Dynkin
diagram of $B_n$ without the last element, then $$W_{\Gamma}(q, t)
= [ m+1 ]_q!,$$ where $m = \mid \Gamma \mid$. If
$\overline{\Gamma}$ is connected and contains the last element of
$B_n$, then by \mbox{\fullref{p:qtpoincare}}
$$W_{\Gamma}(q, t) = [2m]_{q,t}!!, $$ where $m = \mid \Gamma \mid$.

If $\overline{\Gamma}$ is the union of
several connected components of the Dynkin diagram, $\overline{\Gamma}
= \overline{\Gamma}_1
\cup \cdots \cup \overline{\Gamma}_k$, then $W_\Gamma(q,t)$
is the product $$\prod_{i=1}^k W_{\Gamma_i}(q, t)$$ of the factors
corresponding
to the different components.

If $j \notin \Gamma$ we can write $\overline{\Gamma}(j)$ for the connected
component of $\overline{\Gamma \cup \{j\}}$ containing $j$. Suppose that $m =
\mid \Gamma(j) \mid$ and $i$ is the number of elements in $\Gamma(j)$
greater than $j$. Then, if $n \in \Gamma(j)$ we have
$$
\frac{W_{\Gamma \cup \{ j \}}(q,t)} {W_{ \Gamma}(q,t)} = \qbin{m}{i}_{q,t}'
$$
and
$$
\frac{W_{\Gamma \cup \{ j \}}(q,t)} {W_{ \Gamma}(q,t)} =
\qbin{m+1}{i+1}_{q}
$$
otherwise.

\begin{proof} [Sketch of proof of \fullref{t:cohomqt}]
It is convenient to represent generators $\Gamma \sst I_n$ by their
characteristic functions $I_n \to \{0,1\}$ so, simply by strings of
$0$'s and $1$'s of length $n$.

We define a decreasing filtration $F$ on the complex $(C^*_n,\delta)$:
$F^sC_n$ is the
subcomplex generated by the strings of type $A1^s$ (ending with a
string of $s$ $1$s) and we have the inclusions
$$
C_n = F^0C_n \tss F^1C_n \tss \cdots \tss F^nC_n = R.1^n \tss F^{n+1}C_n = 0.
$$
We have the following isomorphism of complexes:
\begin{equation} \label{e:shift}
(F^sC_n/F^{s+1}C_n) \simeq \CA_{n-s-1} [s]
\end{equation}
where $\CA_{n-s-1}$ is the complex for $G_{A_{n-s-1}}$ and the notation $[s]$
means that the degree is shifted by $s$.

The proof uses the spectral sequence $E_*$ associated to the
filtration $F$. The equality \eqref{e:shift} tells us how the
$E_1$ term of the spectral sequence looks like. In fact for $0
\leq s \leq n-2$ we have
\begin{equation} \label{e:sseq1}
E_1^{s,r} = H^{r}(G_{A_{n-s-1}}, R_{q,t}) = H^{r}(G_{A_{n-s-1}},
\Q[q^\pmu]_q)[t^\pmu]
\end{equation}
since the $t$--action is trivial. For $s = n-1$ and $s = n$ the only
non trivial elements in the spectral sequence are
\begin{equation} \label{e:sseq2}
E_1^{n-1,0} = E_1^{n,0} = R.
\end{equation}
If we write $\{m \}[t^\pmu]$ for the module $R/(\ph_m(q))$, then  the
$E_1$--term of
the spectral sequence has a module $\{m \}[t^\pmu]$ in position $(s,r)$
if and only if one of the following condition is satisfied:

a)$m \mid n-s-1$ and $ r = n-s-2 \frac{n-s-1}{m}$;

b)$m \mid n-s$ and $ r=n-s+1-2(\frac{n-s}{m})$.

\ni We know the generators of these modules from \cite{DPS}.
Moreover (see formula \ref{e:sseq2}) we have modules $R$ in
position $(n-1,0)$ and $(n,0)$. The differentials are expressed in
terms of modified binomials defined in formula \ref{e:binqt}. Then
the proof is obtained by a subtle analysis of such differentials.
\end{proof}

%%%%%%%%%%%%%%%%%%%%%%%%%%%%%%%%%%%%%%%%%%%%%%%%%
\section{Some consequences}
%%%%%%%%%%%%%%%%%%%%%%%%%%%%%%%%%%%%%%%%%%%%%%%%%
\fullref{t:cohomqt} gives the cohomology of $G_{B_n}$ as well as
(using \fullref{prop:sha1}) that of $G_{\tilde{A}_{n-1}}$ if we
consider it only as $\Q[q^{\pm 1}]$--module. As regards the rational
cohomology, \fullref{prop:sha1} translates into the following:

\begin{prop}\label{prop:shaq}
We have
$$
H_*(G_{\tilde{A}_{n-1}}, \Q)\cong H_*(G_{B_{n}}, \Q[t^{\pm 1}])
$$
$$
H^*(G_{\tilde{A}_{n-1}}, \Q)\cong H^*(G_{B_{n}}, \Q[[t^{\pm 1}]])
$$
where the action of $G_{B_n}$ on $\Q[t^{\pm 1}]$ (and on
$\Q[[t^{\pm 1}]]$) is trivial for the generators $\epsilon_1,
\ldots, \epsilon_{n-1}$ and $(-t)$--multiplication for the last
generator $\overline{\epsilon}_n$.
\end{prop}

The cohomology of $G_{B_n}$ over the module $\Q[t^{\pm 1}],$ with
action as in \fullref{prop:shaq}, is computed by the complex
$C_n^*$ of \fullref{Section3} where we specialize $q$ to $-1.$ So we use
similar filtration and associated spectral sequence. Recall that the
$\Q$ cohomology of the braid group is of rank $1$ in dimension $0,\
1,$ and vanishes elsewhere.  Then by using a formula analog to
\eqref{e:sseq1} we get
$$
\begin{array}{cccl} \label{e:ssQ}
E_1^{s,r} & = & \Q[t^{\pmu}]\ & \text{if}\quad 0\leq s\leq n,\ r=0\quad \text
{or} \quad 0\leq s\leq n-2,\ r=1\\
          &&& \\
          & = & 0  & \text{otherwise}
\end{array}$$
Next from formula \eqref{e:bordo} it follows
$$
d_1^{s,r}\ =\ \{[s+1]_q\ (1+q^s t)\}_{\{q=-1\}},\ r=0,1
$$
so $d_1^{s,r}=0$ for odd $s$ while $d_1^{s,r}=1+t$ for even $s.$ It
follows that in $E_2$ the odd columns are obtained from the same
columns of $E_1$ dividing by $1+t.$ The even columns vanish, except
for $n$ even it remains
$$E_2^{n-2,1}=E_2^{n,0}=\Q[t^{\pmu}].$$
The only possible non vanishing boundaries are
$$d_2^{s,1}:\ E_2^{s,1}\ \to\ E_2^{s+2,0}$$
and these are of the form
$$
\qbin{s+2}{s}_{[q=-1],t}'.
$$
Up to an invertible, the latter holds $(1+t)(1-t).$ Then $d_2$
vanishes except for $d_2^{n-2,1}$ in case $n$ even. It follows
that:

\begin{teo}\label{teo:ratio}\def\strutt{\vrule width 0pt depth 7pt} One has
$$\begin{array}{cclc}
H^k(G_{B_n},\Q[t^{\pmu}]) & = & \Q[t^{\pmu}]/(1+t) & \quad \ 1\leq
k\leq n-1 \strutt\\
H^n(G_{B_n},\Q[t^{\pmu]}) & = & \Q[t^{\pmu}]/(1+t) & \quad \text{for odd $n$}\strutt\\
H^n(G_{B_n},\Q[t^{\pmu}]) & = & \Q[t^{\pmu}]/(1-t^2) & \quad \text{for
even $n$.}
\end{array}$$
\end{teo}

To obtain the rational cohomology of $G_{\tilde{A}_{n-1}}$ we need to
apply the degree shift in \fullref{prop:shift}.

Notice how \fullref{prop:sha2} changes in the present
situation.

\begin{prop}\label{prop:shaq2}
We have
\begin{align*}
    H_*(G_{B_n}, \Q[t^{\pm 1}])&\cong H_*(G_{A_n}, V)\\
    H^*(G_{B_n}, Q[t^{\pm 1}])&\cong H^*(G_{A_n}, V)
\end{align*}
where  $V$ is the representation of $G_{A_n}$ defined in
\ref{df:tong}.
\end{prop}

As a consequence we have:

\begin{cor} Let $V$ be the $(n+1)$--dimensional representation of the
braid group $Br_{n+1}$ defined in \ref{df:tong}. Then the
cohomology
$$H^*(Br_{n+1};\ V)$$
is given as in \fullref{teo:ratio}.
\end{cor}

%%%%%%%%%%%%%%%%%%%%%%%%%%%%%%%%%%%%%%%%%%%%%%%%%
\section{Related topological constructions}
%%%%%%%%%%%%%%%%%%%%%%%%%%%%%%%%%%%%%%%%%%%%%%%%%
In \cite{S2} the orbit space of any Artin group of finite type,
which is known to be a $K(\pi,1)$ space \cite{del}, was shown to
contract over an explicit  polyhedron with explicit
identifications on its faces (a construction based on \cite{S1}
applied to Coxeter arrangements). As already suggested, few
modifications are needed to obtain a similar description of the
orbit space for Artin groups of infinite type (see also \cite{CD}
for a different construction).

We briefly resume this construction.

Let $(W,S)$ be a (finitely generated) Coxeter group, which we
realize through the Tits representation as a group of (in general,
non orthogonal) reflections in $\R^n,$ where the base-chamber
$C_0$ is the positive octant and $S$ is the set of reflections
with respect to the coordinate hyperplanes. (It is possible to
consider more general representations; see Vinberg \cite{vin}). Let
$U:=W.\overline{C}_0$ be the orbit of the closure of the base
chamber (the {\it Tits cone}). Recall from \cite{vin} that:

\begin{itemize}
    \item
            $U$ is a convex cone in $\R^n$ with vertex $0.$
    \item
            $U=\R^n$ iff $W$ is finite.
    \item
            $U^0:=int(U)$ is open in $\R^n$ and a (relative open) facet
            $F\subset \overline{C}_0$ is contained in $U^0$ iff the stabilizer
            $W_F$ is finite.
\end{itemize}

Let $\mathcal A$ be the arrangement of reflection hyperplanes of $W.$
Set

$$M({\mathcal A})\ :=\ [U^0\ +\ i\R^n]\ \setminus\ \bigcup_{H\in {\mathcal
   A}}\ H_{\C}$$
as the complement of the complexified arrangement. Notice that the
  group $W$ acts freely on $M({\mathcal A})$ so we can consider the {\it
  orbit space}
$$M({\mathcal A})_W:= M({\mathcal A})/W.$$ The associated Artin group
$G_W$ is the fundamental group of the orbit space (see Brieskorn
\cite{Br}, D{\~u}ng \cite{Ng} and van der Lek \cite{van}).

Now take one point $x_0\in C_0;$ for any subset $J\subset S$ such
that the parabolic subgroup $W_J$ is finite, construct a
$|J|$--cell in $U^0$ as the ``convex hull'' of the $W_J$--orbit of
$x_0$ in $\R^n.$

%%%%%%%%%%%%%%%%%%%%%%%%%%%%%%%%%%%%%%%%%%%%%%%%%
\begin{figure}[ht!]
\begin{center}
\includegraphics[scale=.8]{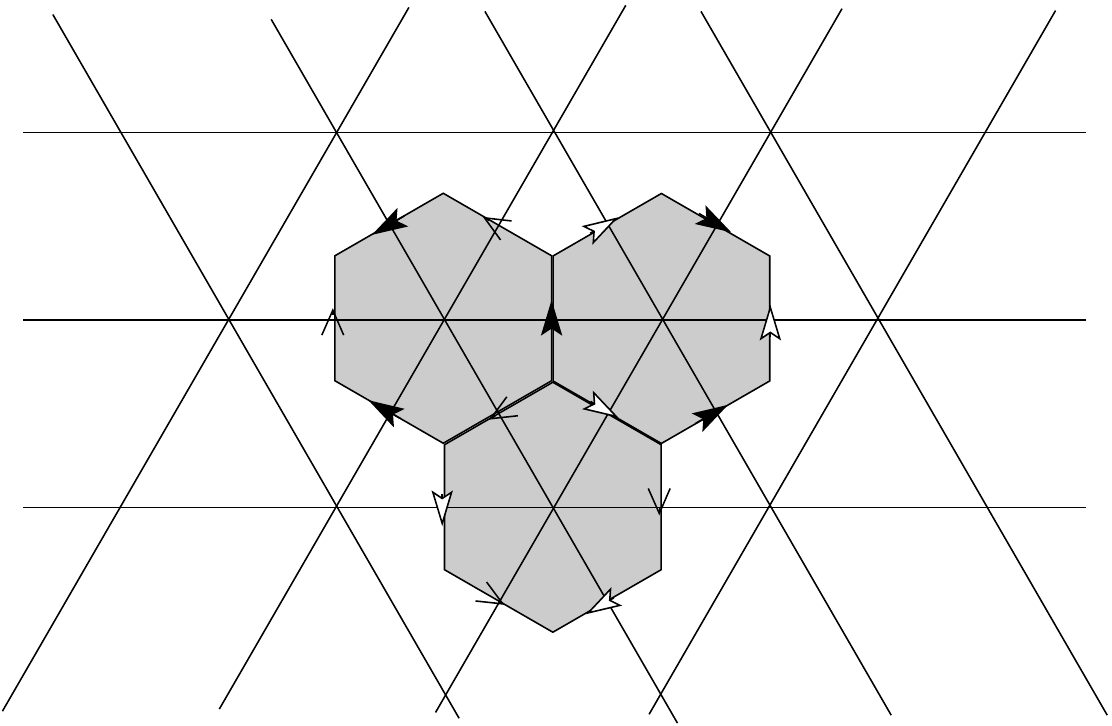}
\caption{{\it the space} $K(G_{\tilde{A_2}},1)$ {\it is given as union of $3$
  hexagons with edges glued according to the arrows (there are: $1$ 0--cell, $3$
  1--cells, $3$ 2--cells in the quotient).}}
\label{fig:affine}
\end{center}
\end{figure}
%%%%%%%%%%%%%%%%%%%%%%%%%%%%%%%%%%%%%%%%%%%%%%%%%

So, we obtain a finite cell complex (see \fullref{fig:affine})
which is the union of (in general, different dimensional) polyhedra,
corresponding to the
maximal subsets $J$ such that $W_J$ is finite. Now take
identifications on the faces of these polyhedra, the same as described
in \cite{S2} for the finite case (they are shown in
\fullref{fig:affine}
for the
case $\tilde{A}_2$). We obtain a finite CW--complex $X_W:$ it has a
$|J|$--cell for each $J\subset S$ such that $W_J$ is finite.

We obtain as in \cite{S2}

\begin{teo} $X_W$ is a deformation retract of the orbit space. \end{teo}

\begin{rem} When $W$ is an affine group, the orbit space is known
to be a $K(\pi,1)$ for types $\tilde{A}_n,\ \tilde{C}_n$ (see
\cite{Ok,CP,CD} for further classes).
\end{rem}

\begin{rem} The standard presentation for $G_W$ is quite easy to
  derive from the topological description of $X_W$; we may thus recover Van del Lek's result \cite{van}.
\end{rem}

\begin{prop} Let $K_W^{fin}:=\{J\subset S\ :\ |W_J|<\infty\}$ with the
  natural structure of simplicial complex. Then the Euler
  characteristic of the orbit space (so, of the
  group $G_W$ when such space is of type $k(\pi,1)$)) equals
$$\chi(K_W^{fin}). $$
In particular, if \ $W$\  is affine of rank \ $n+1$\  we have
$$\chi(M({\mathcal A})_W)\ =\ \chi(K_W^{fin})=\  1-\chi(S^{n-1})\ =\ (-1)^n $$
\end{prop}

\begin{proof} Last statement follows from the fact that $K_W^{fin}$ contains all proper subsets
  of $S$; thus:
$$H_*(K_W^{fin})\ =\ \tilde{H}_{*-1}(S^{n-1}).\proved $$
\end{proof}

\begin{rem} The cohomology of the orbit space in case
  $\tilde{A}_n$ with trivial coefficients is deduced from \fullref{prop:shaq} and from \fullref{teo:ratio}; that with local
coefficients in the $G_{\tilde{A}_{n}}$--module $\Q[q^\pmu]$ is deduced
from \fullref{t:cohomqt}.
\end{rem}

\bibliographystyle{gtart}
\bibliography{link}

\end{document}